\documentclass{amsart}[12pt]
\def\doctype{}

\usepackage{latexsym,amssymb}
\usepackage{tikz}
\usetikzlibrary{decorations.pathreplacing}
\usepackage{color}
\usepackage{fancyhdr}
\usepackage{hyperref}

\newcommand\cX{\mathfrak{X}}

\newcommand\vj{\mathbf{1}}
\newcommand\vb{\mathbf{b}}

\newcommand\vx{\mathbf{x}}
\newcommand\vy{\mathbf{y}}
\newcommand\vz{\mathbf{z}}
\newcommand\vo{\mathbf{0}}

\newcommand\C{\mathbb{C}}

\newcommand\R{\mathbb{R}}

\newcommand{\comment}[1]{}

\numberwithin{equation}{section}

\fancypagestyle{titlepage}{
\fancyhead[R]{\doctype}
\fancyhead[CL]{}
\cfoot{\vspace{5pt} \thepage}
}

\let\oldsection\section
\newcommand\boldsection[1]{\oldsection{\bf #1}}
\newcommand\starsection[1]{\oldsection*{\bf #1}}
\makeatletter
\renewcommand\section{\@ifstar\starsection\boldsection}
\makeatother

\newtheoremstyle{theorem}
  {12pt}		  
  {0pt}  
  {\sl}  
  {\parindent}     
  {\bf}  
  {. }    
  { }    
  {}     
\theoremstyle{theorem}
\newtheorem{thm}{Theorem}[section]  
\newtheorem{lemma}[thm]{Lemma}     
\newtheorem{cor}[thm]{Corollary}

\newtheorem{prop}[thm]{Proposition}

\newtheoremstyle{definition}
  {12pt}		  
  {0pt}  
  {}  
  {\parindent}     
  {\bf}  
  {. }    
  { }    
  {}     
\theoremstyle{definition}

\newtheorem{ex}[thm]{Example}

\newcommand\rk{{\sc Remark.} }

\renewcommand{\proofname}{Proof}

\makeatletter
\renewenvironment{proof}[1][\proofname]{\par
  \pushQED{\qed}%
  \normalfont \partopsep=\z@skip \topsep=\z@skip
  \trivlist
  \item[\hskip\labelsep
        \scshape
    #1\@addpunct{.}]\ignorespaces
}{%
  \popQED\endtrivlist\@endpefalse
}
\makeatother

\makeatletter
\renewcommand*\@maketitle{%
  \normalfont\normalsize
  \@adminfootnotes
  \@mkboth{\@nx\shortauthors}{\@nx\shorttitle}%
  \global\topskip42\p@\relax 
  \@settitle
  \ifx\@empty\authors \else {\vskip 1em
\vtop{\centering\shortauthors\@@par}} \fi
  \ifx\@empty\@date \else {\vskip 1em \vtop{\centering\@date\@@par}}\fi 
  \ifx\@empty\@dedicatory
  \else
    \baselineskip18\p@
    \vtop{\centering{\footnotesize\itshape\@dedicatory\@@par}%
      \global\dimen@i\prevdepth}\prevdepth\dimen@i
  \fi
  \@setabstract
  \normalsize
  \if@titlepage
    \newpage
  \else
    \dimen@34\p@ \advance\dimen@-\baselineskip
    \vskip\dimen@\relax
  \fi
} 
\renewcommand*\@adminfootnotes{%
  \let\@makefnmark\relax  \let\@thefnmark\relax
  \ifx\@empty\@subjclass\else \@footnotetext{\@setsubjclass}\fi
  \ifx\@empty\@keywords\else \@footnotetext{\@setkeywords}\fi
  \ifx\@empty\thankses\else \@footnotetext{%
    \def\par{\let\par\@par}\@setthanks}%
  \fi
\thispagestyle{titlepage}
}
\makeatother

\begin{document}

\title[Fractional decomposition]{\large Fractional triangle decompositions
of dense\\ 3-partite graphs}

\author{Flora C. Bowditch and Peter J.~Dukes}
\address{
Mathematics and Statistics,
University of Victoria, Victoria, Canada
}
\email{bowditch@uvic.ca, dukes@uvic.ca}

\date{\today}

\begin{abstract}
We compute a minimum degree threshold sufficient for 3-partite graphs to
admit a fractional triangle decomposition.  Together with recent work of
Barber, K\"uhn, Lo, Osthus and Taylor, this leads to bounds for exact decompositions and in particular the
completion problem for sparse partial latin squares.  Some extensions are considered as well.
\end{abstract}

\thanks{Research of the authors is supported by NSERC: Flora Bowditch by an NSERC USRA and Peter Dukes by NSERC grant 312595--2010}

\maketitle
\hrule

\bigskip

\section{Introduction}
\label{sec-intro}

\subsection{Decompositions}

A graph $G$ has an $F$-\emph{decomposition} if its edges $E(G)$ can be
partitioned into graphs, each isomorphic to $F$.
Graph decompositions connect with a rich class of problems in combinatorics.
For example, decomposition of a complete graph into cliques is equivalent to
a special case of (pairwise balanced) block designs.  Related topics include
graph labellings, hypergraph matchings, and finite geometries.

For $G$ to admit an $F$-decomposition, it is necessary that $|E(G)|$ be
divisible by $|E(F)|$.  Moreover, the degree of every vertex in $G$ must be
divisible by $\gcd\{\deg(x):x \in V(F) \}$.  A graph $G$ satisfying these
two conditions is said to be $F$-\emph{divisible}.  A conjecture of
Nash-Williams \cite{NW} asks if all $K_3$-divisible graphs $G$ of order $n$
with $\delta(G) \ge \frac{3}{4}n$ admit a $K_3$-decomposition.  Although
this is presently still open, it was recently shown to be true for all very
large graphs $G$ if $\frac{3}{4}$ is replaced by something a bit larger.

\begin{thm}[from \cite{BKLO} and \cite{Dross}]
\label{dross}
Let $\epsilon>0$.  Every $K_3$-divisible graph $G$ on $n \ge n_0(\epsilon)$
vertices with $\delta(G) \ge (\frac{9}{10}+\epsilon) n$ has a
$K_3$-decomposition.
\end{thm}

A key ingredient in the proof of Theorem~\ref{dross} is a result of Barber,
K\"uhn, Lo and Osthus in \cite{BKLO} which connects the minimum degree
threshold for $F$-decompositions to the minimum degree threshold for
approximate $F$-decompositions.  The technique uses absorbers to iteratively
improve approximate decompositions.  Since good approximate decompositions
are implied by `fractional' decompositions, \cite{HR}, the latter has
attracted increased interest.  Formally, we say that $G$ has a
\emph{fractional} $F$-\emph{decomposition} if there is a list of ordered
pairs $(F_i,w_i)$, where $F_i$ is a copy of $F$ in $G$, $w_i$ is a
nonnegative real weight, and such that, for each edge $e \in E(G)$,
$$\sum_{i:e \in E(F_i)} w_i = 1.$$
Note that $F$-divisibility is no longer needed, but there remain `convex
geometric barriers' for the fractional relaxation.  On a related point, there exist counterexamples 
to weakening the minimum degree assumption in Nash-Williams' conjecture.  For instance, the lexicographic
product $C_4 \cdot K_{6m+3}$ is $K_3$-divisible with minimum degree near
$3n/4$, but it violates a geometric barrier for $K_3$-decomposition: namely, there exists a vertex partition
with too many crossing edges.  For similar reasons, a minimum degree hypothesis is generally assumed for fractional decompositions in general.
In particular, the fractional version of Theorem~\ref{dross} is the main result of \cite{Dross}.

Early work on minimum degree thresholds for fractional decompositions was 
done by Yuster, \cite{Yuster1,Yuster2}.  Even in the case of hypergraphs, significant
progress has been made in lowering the thresholds, first by the second
author in \cite{Dukes-corr} and then by Barber, K\"uhn, Lo, Montgomery and
Osthus in \cite{BKLMO}.  We refer the reader to these papers for precise thresholds, which are now qualitatively not too far from best possible.

\subsection{Decompositions in the partite setting}

The authors and those of \cite{BKLMO} recently wondered whether there could
be a `multipartite' analog of the preceding theory.  This sort of variant
has prior precedent; for example, Keevash and Mycroft in \cite{KM} establish
a multipartite version of the Hajnal-Szemer\'edi theorem.  Moreover, the
decomposition of complete $r$-partite graphs into cliques $K_r$ implies
results on the classical problem of mutually orthogonal latin squares just
as clique decompositions of complete graphs relate to block designs.

With this in mind, Barber et al.~very recently produced an $r$-partite
version of their approximate-to-exact result for $K_r$-decompositions.  We
give some remarks before the statement.  Suppose $G$ is $r$-partite and we
seek a $K_r$-decomposition of $G$.  It is necessary that every vertex of $G$
has the same number of neighbors in each of the other $r-1$ partite sets.
We call such a graph (equipped with vertex partition) \emph{locally
balanced}; the term $K_r$-\emph{divisible} is used instead in \cite{BKLOT}.
In any case, it is clear that this condition is actually necessary for the
fractional relaxation as well.
Let $\tau_F(n,r)$ denote the infimum over all $\tau$ such that every locally
balanced $r$-partite graph $G$ on $rn$ vertices with $\delta(G) \ge \tau
(r-1)n$ admits a fractional $K_r$-decomposition.  Put $\tau_F(r):= \lim
\sup_{n \rightarrow \infty} \tau_F(n,r)$.  The approximate-to-exact result
is as follows.

\begin{thm}[\cite{BKLOT}, Corollary 1.6]
\label{integralize}
Let $r$ be an integer at least $3$, $\epsilon > 0$ and let $n >
n_0(\epsilon)$.  Suppose $G$ is a balanced and locally balanced $r$-partite
graph on $rn$ vertices with $\delta(G) \ge (\tau_F(r)+\epsilon)(r-1)n$.
Then $G$ admits a $K_r$-decomposition.
\end{thm}

Except for some general remarks later, our primary focus here is on
calculating an upper bound on $\tau_F(3)$, thereby obtaining a threshold for
triangles in the 3-partite setting.  Our main result is next.

\begin{thm}
\label{main}
We have $\tau_F(3) < 0.96$.  So, for sufficiently large $n$, every
locally balanced 3-partite graph $G$ on $3n$ vertices satisfying
$\delta(G)/2n \ge 0.96$ admits a fractional $K_3$-decomposition.
\end{thm}

The idea of the proof is quite natural, and closely follows
\cite{Dukes-frac,Dukes-corr}.  We would like to choose nonnegative weights
for triangles in $G$ so that the sum of weights at each edge is a constant.
This is a system of linear equations with coefficients in $\{0,1\}$.  For
the complete 3-partite graph $K_{n,n,n}$, the coefficient matrix is rich
with symmetry.  In particular, every edge is included in $n$ triangles, and
so our system has a (positive) constant solution.  Now, since $G$ is close
to $K_{n,n,n}$ and a solution for the latter is well-separated from the zero
vector, there remains a positive solution for $G$.
Some care must be taken in the quantitative analysis, since perturbations of
ill-conditioned linear systems can lead to wild changes in solutions.
Fortunately, the coefficient matrix for $K_{n,n,n}$ is `nice', and in fact
has some interesting connections with the classical matrices of algebraic
coding theory.

Suppose $G$ is a graph that admits a $K_3$-decomposition.  Such a
decomposition can be lifted to the 3-partite graph $G \times K_3$ by
replacing triangle $\{x,y,z\}$ in $G$ with six triangles
$\{(x,i),(y,j),(z,k)\}$ in $G \times K_3$ for $\{i,j,k\}=\{1,2,3\}$.
The converse holds fractionally.  That is, any fractional
$K_3$-decomposition of $G \times K_3$ projects to a fractional
$K_3$-decomposition of $G$ by averaging over the six pre-images of
$\{x,y,z\}$. It follows that determining an upper bound on the 3-partite
degree threshold $\tau_F(3)$ is at least as difficult as obtaining the same
threshold bound for the ordinary dense case.  There is no obvious way of
using Theorem~\ref{dross} to conclude anything about the $3$-partite
threshold $\tau_F(3)$.

\subsection{Latin squares}
A \emph{latin square} of order $n$ is an $n \times n$ array with entries
from a set of $n$ symbols (often taken to be $[n]:=\{1,2,\dots,n\}$) having
the property that each symbol appears exactly once in every row and every
column.  Naturally, a \emph{partial latin square} of order $n$ is an $n
\times n$ array whose cells are either empty or filled in such a way that
each symbol appears at most once in every row and every column.
A partial latin square can be identified with a set of ordered triples in a
natural way: if symbol $k$ appears in row $i$ and column $j$, we include the
ordered triple $(i,j,k)$.  A \emph{completion} of a partial latin square $P$
is a latin square $L$ which contains $P$ in the sense of ordered triples;
that is, we have the filled cells of
$P$ agreeing with corresponding cells in $L$.

A latin square of order $n$ is equivalent both to a one-factorization of
$K_{n,n}$ and also a $K_3$-decomposition of $K_{n,n,n}$.  The latter
corresponds with the representation by ordered triples, where the three
partite sets are rows, columns, and symbols.  Similarly, a partial latin
square $P$ corresponds to an edge-disjoint union of triangles  in
$K_{n,n,n}$, and the completion problem amounts to a $K_3$-decomposition of
the (3-partite) complement.

Following Bartlett, we call a partial latin square (of order $n$)
$c$-\emph{dense} if every row, column, and symbol appears at most $c n$
times.  For the completion problem, such latin squares induce locally
balanced 3-partite graphs with minimum degree at least $2(1-c)n$.
Daykin and H\"aggkvist conjectured in \cite{DH} that all $1/4$-dense partial
latin squares can be completed.  The first serious progress toward this
conjecture was by Chetwynd and H\"aggkvist, who showed in \cite{CH} that,
for sufficiently large even integers $n$, $c=10^{-5}$ suffices to guarantee
a completion.
Gustavsson \cite{Gus} obtained the threshold $c=10^{-7}$ for all $n$.  These
proofs were technical and required long chains of substitutions.  Recently,
Bartlett \cite{Bartlett} increased the threshold to $c=10^{-4}$ for large
$n$ using the notion of `improper trades'.  In fact, this method showed that
completion is possible for densities near $1/12$, but under a strong
additional assumption on the total number of filled cells.

Improving the threshold on $c$ is one consequence of
Theorem~\ref{integralize} and our main result, Theorem~\ref{main}.

\begin{cor}[see also \cite{BKLOT}]
\label{ls-completion}
Let $0<c \le 0.04$.  Let $P$ be a partial latin square of order $n \ge n_0(c)$ such that  every row,
column, and symbol is used at most $cn$ times.
Then $P$ can be completed to a latin square.
\end{cor}

To prove Corollary~\ref{ls-completion}, it suffices to find a
$K_3$-decomposition of the graph $G_P$ on $3n$ vertices, one for each row,
column, and symbol, where an edge is drawn between two vertices if and only if they are
not incident in $P$.
By construction, $G_P$ is locally balanced, and $\delta(G_P) \ge 2(1-c)n$.  
By Theorem~\ref{integralize} with our $\tau_F(3) < 0.96$, it follows that $G_P$ admits a $K_3$-decomposition.

\subsection{Organization of the paper}

In the next section, we set up a linear system that models triangle decompositions of graphs in the dense 3-partite
setting.  In Section~\ref{sec-norms}, our approximation technique is made
precise, allowing us to turn our attention to the system for $K_{n,n,n}$.
Then, in Section~\ref{sec-assoc}, the coefficient matrix for $K_{n,n,n}$ is shown to lie in a certain low-dimensional algebra.
This permits the calculations necessary to complete the proof of Theorem~\ref{main}, which occurs in Section~\ref{sec-pf}. We
sketch the technique for larger cliques and hypergraphs in Section~\ref{sec-generalize},
and conclude with a discussion of possible future directions in Section~\ref{sec-discuss}.

\section{The linear system}
\label{sec-sys}

Let $G$ be a graph.  We work primarily in the vector space
$\Omega(G):=\mathbb{R}^{E(G)}$, whose coordinates are indexed by edges of
$G$ (in any order).  Let $T(G)$ be the set of all triangles in $G$ and let
$W_G$ be the $\{0,1\}$ `inclusion' matrix of $E(G)$ versus $T(G)$.  That is,
$W_G$ has rows indexed by $E(G)$, columns indexed by $T(G)$, and where
$$W_G(e,t) = \begin{cases}
1 & \text{if } e \subseteq t, \\
0 & \text{otherwise}.
\end{cases}$$

A fractional $K_3$-decomposition of $G$ is equivalent to a solution
$\mathbf{z} \ge \mathbf{0}$, by which we mean that $\mathbf{z}$ is
entrywise nonnegative, to the system
\begin{equation}
\label{system}
W_G \mathbf{z}=\mathbf{1},
\end{equation}
where $\vj$ is the vector of all ones in $\Omega(G)$.

In general, some edges of $G$ might belong to no triangles, and so
(\ref{system}) might have no solution.  Even if there is a solution, there
are usually more triangles than edges and such a solution is not unique.

Following \cite{Kseniya}, the set $\{t \in T(G): t \supseteq e \}$ is called
the \emph{fan} in $G$ at $e$.
If the fans at each edge in $G$ are very rich, it is reasonable to ask
whether we can decompose $G$ into fans; this corresponds to assuming that
our solution $\vz$ is a linear combination of the rows of $W_G$, since the
rows of $W_G$ tell us which triangles contain a given edge.  The linear
system for fan decomposition is then
\begin{equation}
\label{system-fan}
M_G \vx=\vj,
\end{equation}
where $M_G = W_G W_G^\top$ is a square matrix.  Combinatorially, $M_G$ has
rows and columns indexed by $E(G)$, and the $(e,f)$-entry of $M_G$ records
the number of triangles in $G$ containing $e \cup f$.  Observe that a
solution $\mathbf{x} \ge \mathbf{0}$ to (\ref{system-fan}) implies a
solution $\mathbf{z} \ge \mathbf{0}$ to (\ref{system}) and therefore implies
a fractional triangle decomposition of $G$.

If $M_G$ is non-singular, then of course the system (\ref{system-fan}) has a
unique solution.  In general, though, $M_G$ may have nontrivial kernel; we
describe this kernel for the dense 3-partite case later.  First, we offer
some examples of $M_G$ to illustrate the method.

\begin{ex}
If $G=K_n$, then $M_G = (n-2)I + A_1$, where $A_1$ denotes the adjacency
matrix for the line graph of $K_n$.  The eigenvalues of $A_1$ are well known
to be $2n-4$, with multiplicity $1$, $v-4$, with multiplicity $2n-4$, and
$-2$, with multiplicity $\binom{n-2}{2}$; see \cite{GodsilRoyle}, for
example.  In fact, the eigenspaces of $A_1$ admit a nice description.  Since
$A_1$ is symmetric, this description affords an explicit orthogonal
decomposition of $\Omega(K_n)$.  It follows that $M_G^{-1}$ exists and can
be computed explicitly for all $n$.  In any case, the unique solution to
(\ref{system-fan}) is $\vx = \frac{1}{3n-6} \vj$, the eigenvector for the
largest eigenvalue.
\end{ex}

Our new investigation starts with the case $G=K_{n,n,n}$.
Since we work with this graph frequently in what follows, we suppress subscripts on $W$ and $M$ for this
graph.  That is, $W$ is the inclusion matrix of $E(K_{n,n,n})$ versus
$T(K_{n,n,n})$ and 
$$M_G(e,f) = \begin{cases}
n & \text{if } e =f, \\
1 & \text{if } e \cup f \text{ consists of three points in different partite classes,}\\
0 & \text{otherwise}.
\end{cases}$$
Similar to the case of $K_n$, we have $M$ agreeing with the line graph of
$K_{n,n,n}$ off the diagonal.

\begin{ex}
With $n=2$ and rows/columns organized by partite sets, we have
$$M=
\begin{footnotesize}
\left[
\begin{array}{cccc|cccc|cccc}
2 & 0 & 0 & 0 & 1 & 1 & 0 & 0 & 1 & 1 & 0 & 0\\
0 & 2 & 0 & 0 & 1 & 1 & 0 & 0 & 0 & 0 & 1 & 1\\
0 & 0 & 2 & 0 & 0 & 0 & 1 & 1 & 1 & 1 & 0 & 0\\
0 & 0 & 0 & 2 & 0 & 0 & 1 & 1 & 0 & 0 & 1 & 1\\
\hline
1 & 1 & 0 & 0 & 2 & 0 & 0 & 0 & 1 & 0 & 1 & 0\\
1 & 1 & 0 & 0 & 0 & 2 & 0 & 0 & 0 & 1 & 0 & 1\\
0 & 0 & 1 & 1 & 0 & 0 & 2 & 0 & 1 & 0 & 1 & 0\\
0 & 0 & 1 & 1 & 0 & 0 & 0 & 2 & 0 & 1 & 0 & 1\\
\hline
1 & 0 & 1 & 0 & 1 & 0 & 1 & 0 & 2 & 0 & 0 & 0\\
1 & 0 & 1 & 0 & 0 & 1 & 0 & 1 & 0 & 2 & 0 & 0\\
0 & 1 & 0 & 1 & 1 & 0 & 1 & 0 & 0 & 0 & 2 & 0\\
0 & 1 & 0 & 1 & 0 & 1 & 0 & 1 & 0 & 0 & 0 & 2
\end{array}
\right].
\end{footnotesize}
$$
\end{ex}
It is not hard to compute the rank of $W$ (and of $M$).

\begin{prop}
\label{rank}
We have $\text{rank}(W)=\text{rank}(M)=n^3-(n-1)^3$.
\end{prop}

\begin{proof}
It is well known that, for matrices $A$ over $\R$ or $\C$,
$\text{rank}(A)=\text{rank}(AA^\top)$.  This gives the first equality.

Let $\alpha\beta\gamma$ be any triangle in $K_{n,n,n}$. We claim that the
set $\mathcal{I}$ of $n^3-(n-1)^3$ triangles which intersect $\alpha \beta
\gamma$ in at least one point is linearly independent. Consider the edge
$\beta' \gamma' \in E(K_{n,n,n})$ for $\beta',\gamma'$ in the same parts but
distinct from $\beta,\gamma$.  It belongs only to the triangle $\alpha
\beta' \gamma' \in \mathcal{I}$.  Likewise, an edge of the form $\beta'
\gamma$ belongs to a unique triangle among those intersecting $\alpha \beta
\gamma$ in at least two points.  From this argument we obtain
$\text{rank}(W) \ge n^3-(n-1)^3$.

For the reverse inequality, if $\alpha \beta \gamma$ is used to denote the
formal linear combination $\alpha \beta + \alpha \gamma + \beta \gamma \in
\Omega(K_{n,n,n})$, there is the identity
$$\alpha' \beta' \gamma' = \alpha \beta \gamma - \alpha'\beta \gamma  -
\alpha\beta' \gamma - \alpha \beta \gamma' + \alpha'\beta' \gamma +
\alpha'\beta \gamma' + \alpha\beta' \gamma'.$$
This shows that every triangle in $T(K_{n,n,n})$ is a linear combination of
those in $\mathcal{I}$.
\end{proof}

Let us now discuss the kernel of $W^\top$ (and $M$).  By
Proposition~\ref{rank}, we have
$$\dim \ker(W^\top)=\dim \ker(M)=3n^2-(n^3-(n-1)^3) = 3n-1.$$
Let $\beta \in V(K_{n,n,n})$.  As in the proof above, we adopt the
convention that $\alpha,\beta,\gamma$ (and their variants) stand for
vertices in the three different partite sets labelled in some consistent
cyclic order.  With this understanding, define the vector $\mathbf{v}_\beta
\in \Omega(K_{n,n,n})$ by
$$\mathbf{v}_\beta (e) =
\begin{cases}
-1 & \text{if}~e = \alpha \beta~\text{for some}~\alpha,\\
1 & \text{if}~e = \beta \gamma~\text{for some}~\gamma, \\
0 & \text{otherwise}.
\end{cases}$$
It is clear that $\mathbf{v}_\beta$ vanishes on $T(K_{n,n,n})$; therefore
$\mathbf{v}_\beta \in \ker(W^\top)$ for each $\beta$.  Any set of $3n-1$ of
these vectors is linearly independent (their sum vanishes) and forms a basis
for $\ker(W^\top)=\ker(M)$.

We now sketch a `kernel elimination strategy' that is useful for our
problem.  Let $K$ be the matrix which applies orthogonal projection onto
$\ker(M)$.  Then, for any nonzero real number $\eta$, the linear system
$(M+\eta K) \vx = \vj$ has the unique constant solution $\vx = \frac{1}{3n}
\vj$.  This can be viewed alternatively as the addition of $3n-1$ artificial
columns $\mathbf{v}_\beta$ to $W$.  The resulting matrix has full row rank
$3n^2$.

Suppose now that $G$ is a spanning subgraph of $K_{n,n,n}$.  Let $\vj_G \in
\Omega(K_{n,n,n})$ be the characteristic vector of edges in $G$; that is,
\begin{equation*}
\vj_G (e) = \begin{cases}
1 & \text{if } e \in E(G), \\
0 & \text{otherwise}.
\end{cases}
\end{equation*}
Also, for a square matrix $A$ indexed by $\Omega(K_{n,n,n})$, let $A[G]$
denote its restriction to the principal submatrix indexed by $\Omega(G)$.
The kernel elimination strategy for $M_G$ is similar in spirit as that for
$M$.  Here, though, we add a multiple of $K[G]$ and must justify that this
works.
First, the following is easily verified.

\begin{prop}
\label{kernel1}
Let $G$ be a locally balanced spanning subgraph of $K_{n,n,n}$.  Then
$K[G] \vj = \vo$.
\end{prop}

\begin{proof}
Since $G$ is locally balanced, it is orthogonal to every vector in $\ker(M)$.  Therefore, $K \vj_G=\vo$.  The claim follows by restricting to $G$.
\end{proof}

Next, we have an important orthogonality relation.

\begin{prop}
\label{kernel2}
With $K$ and $M_G$ defined as above, $K[G] M_G  = O$.
\end{prop}

\begin{proof}
Let $L$ denote the inclusion map from $\Omega(G)$ to $\Omega(K_{n,n,n})$.
As a matrix, assuming rows are organized, we have
$$L=\left[
\begin{array}{c}
I\\
\hline
O
\end{array}
\right].
$$
Right multiplication by $L$ restricts to columns indexed by $E(G)$.
In particular, $K[G]=L^\top K L$.
Also, if we sort the rows and columns of $W$ so that $E(G)$ and $T(G)$ come
first, then we have
$$W=
\left[
\begin{array}{c|c}
W_G & *\\
\hline
 O & *
\end{array}
\right].
$$
Using these observations, we compute
\begin{equation*}
K[G] M_G = L^\top K L W_G W_G^\top = L^\top K W
\left[
\begin{array}{c}
I_{|T(G)|} \\
\hline
O\\
\end{array}
\right] W_G^\top = O.\qedhere
\end{equation*}
\end{proof}

\rk
Propositions~\ref{rank}, \ref{kernel1} and \ref{kernel2} are straightforward to extend to the setting of
$r$-cliques in $r$-partite graphs and even to hypergraphs.  We omit the details.

The preceding facts feed into the following result, which roughly states
that solutions to a symmetric linear system are unchanged if the coefficient
matrix undergoes an orthogonal shift.
\begin{lemma}
\label{shift}
Let $A$ and $B$ be Hermitian $N \times N$ matrices with $AB=O$ and $A+B$
nonsingular.  Suppose also that $B \vb = \vo$. Then $A(A+B)^{-1} \vb = \vb$.
\end{lemma}

\begin{proof}
The matrices $A$ and $B$ generate a commutative algebra of Hermitian
matrices; hence, they admit a common basis
$\{\mathbf{e}_1,\dots,\mathbf{e}_N \}$ of orthonormal eigenvectors.  For
$i=1,\dots,N$, put $E_i = \mathbf{e}_i\mathbf{e}_i^*$, a rank one
projection.  We have $\sum_{i=1}^N E_i = I$ and $E_iE_j=O$ for $i \neq j$.

Suppose
$A=a_1 E_1+\dots+a_r E_r$ and $A+B=a_1 E_1 + \dots + a_N E_N$ for nonzero
coefficients $a_i$.  Since $B \vb=\vo$, we have $E_j \vb = \vo$ for $r<j \le
N$.  With this, we compute
\begin{equation*}
A(A+B)^{-1} \vb = (a_1 E_1 + \dots + a_r E_r)
(a_1^{-1} E_1 + \dots + a_N^{-1} E_N) \vb = (E_1+\dots+E_r) \vb =
\vb.\qedhere
\end{equation*}
\end{proof}

We apply Lemma~\ref{shift} by putting $A=M_G$, $B=\eta K[G]$ (note that both
are symmetric with real entries), and $\vb=\vj$.   We show later that $M_G+
\eta K[G]$ is nonsingular for $\eta \neq 0$ under our minimum degree
assumption for $G$.  It follows by Propositions~\ref{kernel1} and
\ref{kernel2} that the (unique) solution of $(M_G+\eta K[G]) \vx = \vj$ also
provides a solution of $M_G \vx = \vj$.  The next section develops some
tools to ensure such a solution $\vx$ is nonnegative, thereby giving a
fractional decomposition of $G$ into fans.

\section{Norms}
\label{sec-norms}

Here we review some basic facts concerning vector and matrix norms.  The end goal is
a sufficient condition for certain linear systems to have a nonnegative
solution.  Further details and discussion can be found in \cite[Chapter
5]{MatrixAnalysis}.

A \emph{matrix norm} is a function $|| \cdot ||$ from (complex-valued)
matrices of a given shape to the nonnegative reals satisfying: (1) $||A||=0$
if and only if $A=O$, (2) $||\alpha A||= |\alpha| \; ||A||$ for scalars
$\alpha$, and (3) the triangle inequality $||A+B|| \le ||A|| + ||B||$.

For $\vx \in \C^N$, and $p \ge 1$, recall the vector $p$-\emph{norm}
\begin{equation}
\label{vec-p-norm}
||\vx||_p = \left( \sum_{i=1}^N |x_i|^p \right)^{1/p}.
\end{equation}
With $p=\infty$, we take the special (and limiting) definition
$||\vx||_\infty = \max \{|x_i| : i=1,\dots,N\}$ instead of
(\ref{vec-p-norm}).  By Minkowski's inequality, these are
matrix norms on $N \times 1$ columns for each $p$.

More generally, the matrix norm on $\C^{N \times N}$ \emph{induced by} the
$p$-norm is given by
$$||A||_p := \max_{\vx \neq 0} \frac{||A \vx||_p}{||\vx||_p}.$$
It is straightforward to check that the matrix $p$-norm (in fact any induced
norm) is submultiplicative.
\begin{prop}
\label{submult}
Let $A,B \in \C^{N \times N}$.  Then $||AB||_p \le ||A||_p ||B||_p$.
\end{prop}

Here, we are mainly interested in the special case $||A||_\infty =
\max_i \sum_j |A(i,j)|$, the maximum absolute row sum of $A$.  It is worth mentioning, 
though, that the Euclidean norm on vectors induces $||A||_2$, the largest singular value of $AA^*$.  In
the case that $A$ is real symmetric (or Hermitian), this is simply the spectral radius $\rho(A)$.  
Proposition~\ref{submult} readily implies that $\rho(A)$ is a lower bound on any induced norm.  

\begin{prop}[See \cite{MatrixAnalysis}]
\label{cond}
Let $A \in \mathbb{C}^{N \times N}$ be invertible, and consider the system of equations $A \vx = \vb$.  Suppose $A+\delta A$ is a perturbation with $||A^{-1} \delta A ||_p <1$.  Then $A+\delta A$ is nonsingular and the unique solution $\vx + \delta\vx$ to the equation $(A+\delta A)(\vx + \delta \vx) = \vb$ has relative error
\begin{equation}
\label{est}
\frac{||\delta \vx||_p}{||\vx||_p} \le \frac{||A^{-1} \delta A||_p}{1-||A^{-1} \delta A||_p}.
\end{equation}
\end{prop}
This can be proved by expanding $\delta \vx = (A+\delta A)^{-1} \vb - A^{-1} \vb$ as a geometric series, and applying the triangle inequality.  See \cite[\S 5.8]{MatrixAnalysis} for more details of the proof.

Working from this, we note that the existence of nonnegative solutions to certain square linear systems can be verified using the $\infty$-norm.  Here is the instance we shall use.

\begin{cor}
\label{nonneg}
Suppose a nonnegative constant vector $\vx$ solves the square system $A \vx = \vb$ in Proposition~\ref{cond}.  Then the solution $\vy$ to $(A+\delta A)\vy = \vb$ is entrywise nonnegative if $||A^{-1} \delta A||_\infty \le \frac{1}{2}$.
\end{cor}

\begin{proof}
Without loss of generality, suppose $\vx= \vj$, the all ones vector.  By Proposition~\ref{cond} and our norm assumption,  
$$||\delta \vx||_\infty \le \frac{||A^{-1} \delta A||_\infty}{1-||A^{-1} \delta A||_\infty} \le 1.$$
It follows that the entries of $\vy=\vx+\delta \vx$ are between $0$ and $2$.
\end{proof}

\rk
In view of Proposition~\ref{submult}, the conclusion also holds if $||A^{-1}||_\infty \cdot ||\delta A||_\infty \le \frac{1}{2}$.

In some sense, this is the main engine for our argument.  Recall that in
Section~\ref{sec-sys} we had set up a matrix $A_G=M_G+\eta K[G]$ so that $G$
has a fractional decomposition into fans if and only if $A_G \vx = \vj$ has a
solution $\vx \ge \vo$.  Using Corollary~\ref{nonneg}, our proof amounts to 
upper-bounding two matrix norms: a perturbation (from our mindegree assumption), 
and $A^{-1}$, which can be obtained explicitly.

\section{A Bose-Mesner algebra}
\label{sec-assoc}

The main purpose of this section is to compute the inverse of $A=M+\eta K$,
where recall that $M \in \mathbb{C}^{N \times N}$ is a symmetric matrix counting
triangles in $\Omega(K_{n,n,n})$, $K$ is the orthogonal projection onto
$\ker(M)$, and $\eta \neq 0$ is a real parameter.  This is aided by showing
that $A$ lives in a low-dimensional algebra, which we can compute
explicitly.  We begin with some background.

A symmetric $k$-{\em class association scheme} on a set $\cX$ consists of
$k+1$
nonempty symmetric binary relations $R_0,\dots,R_k$ which partition $\cX
\times \cX$, such that
\begin{itemize}
\item
$R_0$ is the identity relation, and
\item
for any $x,y \in \cX$ with $(x,y) \in R_h$, the number of $z \in \cX$ such
that $(x,z) \in R_i$ and $(z,y) \in R_j$ is the
{\em structure constant}
$a^h_{ij}$
depending only on $h,i,j$.  In particular, each $R_i$ is a regular graph of
degree $\nu_i:= a^0_{ii}$ (with $R_0$ consisting of isolated loops).
\end{itemize}

Let $|\cX|=N$. For $i=0,\dots,k$, define the $N \times N$ {\em adjacency
matrix} $A_i$, indexed by entries of $\cX$, to have $(x,y)$-entry equal to $1$ if $(x,y) \in R_i$, and $0$ otherwise. We say
that $x$ and $y$ are $i$th {\em associates} when $(x,y) \in R_i$.  Since the relations partition $\cX \times \cX$, we have
$A_0+A_1+\dots+A_k =J$, the all-ones matrix.

By definition of the structure constants,
\begin{equation}
\label{struct}
A_i A_j = \sum_{h=0}^k a^h_{ij} A_h.
\end{equation}
In this way, the adjacency matrices span a commutative algebra of symmetric
matrices called the
{\em Bose-Mesner algebra} of $\cX$.  We write $\mathfrak{A}=\langle
A_0,A_1,\dots,A_k \rangle$.

\begin{ex}
The {\em Johnson scheme} $J(k,v)$ has as elements $\binom{[v]}{k}$.  Subsets
$K,L \in \binom{[v]}{k}$ are declared
to be $i$th associates if and only if $|K \cap L|=k-i$.
\end{ex}

\begin{ex}
The \emph{Hamming scheme} $H(k,n)$ has as elements $[n]^k$.  Two such words
are declared to be $i$th associates if and only if their Hamming distance
equals $i$.
\end{ex}

More generally, the \emph{Hamming lattice} has ground set
$\mathcal{H}_{k,n}=([n] \cup \{*\})^k$, elements of which we call
\emph{subwords}.  The partial order $\preceq$ is defined by `inclusion';
that is, $x \preceq y$ if and only if, for all $i$, we have $x_i \in
\{y_i,*\}$.  Then $\mathcal{H}_{k,n}$ is a regular meet semilattice,
\cite{Delsarte-semilattice}.  The \emph{rank} of a subword $x$ is $|\{i: x_i
\neq *\}|$, and the set of subwords of rank $r$ is the $r$th \emph{level} of
the Hamming lattice.

We investigate the Hamming scheme itself in more generality in
Section~\ref{hyper}.  Here, though, we consider the case of triangle
decompositions as a concrete starting point.  The vertices, edges, and
triangles in $K_{n,n,n}$ correspond with the elements of rank $1,2,3$,
respectively, in $\mathcal{H}_{3,n}$.  Our matrix $W$ is simply the
incidence matrix of the second level versus the third level.  Accordingly,
$M=WW^\top$ counts the elements above a given two elements in the second
level.

\begin{prop}
The second level of $\mathcal{H}_{3,n}$ is a symmetric $4$-class association
scheme.
\end{prop}

\begin{proof}
A subword of rank 2, say $\alpha \beta *$, can interact with other subwords
in five essentially distinct ways: $\alpha \beta *$, $\alpha \beta' *$,
$\alpha' \beta' *$, $* \beta \gamma$, or $*\beta' \gamma$.  Here, we mean
for each variable to be unequal to its dashed counterpart.  With these
defining relations $R_0,\dots,R_4$, it is straightforward to compute the
structure constants by counting.  See Table~\ref{table-struct} for a full
list of the nontrivial structure constants (recall $a^h_{ij}=a^h_{ji}$ and
$a^h_{i0}=1$ or $0$ according as $i=h$).
\end{proof}

\begin{table}
\footnotesize
\begin{center}
$
\begin{array}{c|cccc}
a_{ij}^0 & 1 & 2 & 3 & 4\\
\hline
1 & 2n-2 & 0 & 0 & 0 \\
2 &  & (n-1)^2 & 0 & 0 \\
3 & && 2n & 0 \\
4 & &&& n(2n-2)
\end{array}
\hspace{1cm}
\begin{array}{c|cccc}
a_{ij}^1 & 1 & 2 & 3 & 4\\
\hline
1 & n-2 & n-1 & 0 & 0 \\
2 &  & (n-1)(n-2) & 0 & 0 \\
3 & && n & n \\
4 & &&& n(2n-3)
\end{array}$

\smallskip
$
\begin{array}{c|cccc}
a_{ij}^2 & 1 & 2 & 3 & 4\\
\hline
1 & 2 & 2n-4 & 0 & 0 \\
2 &  & (n-2)^2 & 0 & 0 \\
3 & && 0 & 2n \\
4 & &&& n(2n-4)
\end{array}
\hspace{.5cm}
\begin{array}{c|cccc}
a_{ij}^3 & 1 & 2 & 3 & 4\\
\hline
1 & 0 & 0 & n-1 & n-1  \\
2 &  & 0 & 0 & (n-1)^2 \\
3 & && 1 & n-1 \\
4 & &&& (n-1)^2
\end{array}
\hspace{.5cm}
\begin{array}{c|cccc}
a_{ij}^4 & 1 & 2 & 3 & 4\\
\hline
1 & 0 & 0 & 1 & 2n-3  \\
2 &  & 0 & n-1 & (n-1)(n-2) \\
3 & && 1 & n-1 \\
4 & &&& (n-1)^2
\end{array}$
\end{center}
\normalsize
\medskip
\caption{Structure constants for the second level of $\mathcal{H}_{3,n}$}
\label{table-struct}
\end{table}

Let $A'_0,A'_1,\dots,A'_4$ be the adjacency matrices for relations
$R_0,R_1,\dots,R_4$ as described in the proof.
As an example calculation in Table~\ref{table-struct}, we have $A'_1 A'_3 =
(n-1)A'_3+A'_4$: there are exactly $n-1$ elements $\alpha' \beta *$ which
are simultaneously first associates with $\alpha \beta *$ and third
associates with $* \beta \gamma$, and there is exactly one element, namely
$\alpha \beta' *$, which is first associates with $\alpha \beta *$ and simultaneously
third associates with $* \beta' \gamma$.

It is worth highlighting the special case of the degrees $\nu_i$ for this
scheme.

\begin{prop}
\label{degrees}
For the second level of $\mathcal{H}_{3,n}$, degrees are $\nu_0=1$,
$\nu_1=2(n-1)$, $\nu_2=(n-1)^2$, $\nu_3=2n$, $\nu_4=2n(n-1)$.
\end{prop}

Observe that $M=WW^\top = nI + A'_3$.  In other words, in $K_{n,n,n}$, any
edge is contained in exactly $n$ triangles, while any two edges which are
third associates (of the form $\alpha \beta *$ and $* \beta \gamma$) are
contained in exactly one triangle (that being $\alpha \beta \gamma$).

In general, a Bose-Mesner algebra $\mathfrak{A}$ is commutative; see
(\ref{struct}) and the definition of the coefficients.  It follows that
$\mathfrak{A}$ has a common set of eigenspaces, and hence a basis of
orthogonal idempotents. In the case of the second level of
$\mathcal{H}_{3,n}$, the eigenspaces of our $M$ have a natural description.
Since a more thorough and general spectral analysis using the Hamming scheme
appears later, we merely sketch the concrete case for triangle
decompositions.

\begin{prop}
\label{evals}
The nonzero eigenvalues of $M$ are $\theta_0=3n$, $\theta_1=2n$, and
$\theta_2=n$.  Corresponding eigenvectors are given by
\vspace{-12pt}
\begin{itemize}
\item
$\vj$ $($unique up to multiples$)$ for $\theta_0$,
\item
$\sum_{\alpha} (\alpha \beta * - \alpha \beta' *) + \sum_{\gamma} (* \beta
\gamma - * \beta' \gamma)$ $($in total $3(n-1)$ independent vectors$)$ for
$\theta_1$, and
\item
$\alpha \beta *  - \alpha \beta' \mbox{$*$} - \alpha' \beta \mbox{$*$}  +
\alpha' \beta' *$ $($in total $3(n-1)^2$ independent vectors$)$ for
$\theta_2$.
\end{itemize}
\end{prop}

\begin{proof}[Proof sketch]
It is clear that $M \vj = 3n \vj$ since $\alpha \beta *$ extends in $n$ ways
on its own, and defines for each $\gamma$ exactly one common extension with
$\alpha \mbox{$*$} \gamma$ and $* \beta \gamma$.

Next, let $\mathbf{u}_{\beta,\beta'}$ denote the second given vector.  Since
there are exactly $n$ completions of each subword in the second level, we
have
$M \mathbf{u}_{\beta,\beta'} (\alpha \beta*) = 2n$ and $M
\mathbf{u}_{\beta,\beta'} (\alpha \beta'*) = -2n$.  Similar identities hold
for $*\beta \gamma$ and $*\beta' \gamma$, and otherwise $M
\mathbf{u}_{\beta,\beta'}$ vanishes.  So $M  \mathbf{u}_{\beta,\beta'} = 2n
\mathbf{u}_{\beta,\beta'}$ as desired.

From the third vector, we compute $$M(\alpha \beta *  - \alpha \beta' * -
\alpha' \beta * + \alpha' \beta' *) = n(\alpha \beta *  - \alpha \beta' * -
\alpha' \beta * + \alpha' \beta' *)$$
due to cancellation on all but the four given edges.  For instance, on $*
\beta \gamma$, we pick up $+1$ from $\alpha \beta *$ and $-1$ from $\alpha'
\beta *$.

Finally, the dimensions are as stated because of Proposition~\ref{rank} and
some obvious relations on the above vectors.
\end{proof}

Having these eigenspaces, computing the corresponding idempotents is
straightforward.  The key thing to note is that these idempotents live in
$\mathfrak{A}$, so they are linear combinations of the $A'_i$.

\begin{prop}
\label{idempotents}
With $I=A'_0,A'_1,\dots,A'_4$ as described above, orthogonal projections
onto the eigenspaces of $M$ for eigenvalues $\theta_0,\theta_1,\theta_2$
are, respectively
\begin{align*}
E_0 &= \frac{1}{3n^2} A'_0 +  \frac{1}{3n^2} A'_1 +  \frac{1}{3n^2} A'_2 +
\frac{1}{3n^2} A'_3 +  \frac{1}{3n^2} A'_4,\\
E_1 &= \frac{n-1}{n^2} A'_0 + \frac{n-2}{2n^2} A'_1 - \frac{1}{n^2} A'_2 +
\frac{n-1}{2n^2} A'_3 - \frac{1}{2n^2} A'_4,~~\text{and}\\
E_2 &= \frac{(n-1)^2}{n^2} A'_0 - \frac{n-1}{n^2} A'_1  + \frac{1}{n^2} A'_2.
\end{align*}
Orthogonal projection onto the kernel of $M$ is given by $K=I-E_0-E_1-E_2$.
\end{prop}
It is possible, though tedious, to verify Proposition~\ref{idempotents} by a direct computation $E_i
\mathbf{e}_j = \delta_{ij} \mathbf{e}_j$, where
$\mathbf{e}_0,\mathbf{e}_1,\mathbf{e}_2$ are the eigenvectors from
Proposition~\ref{evals}.  But we omit details, since a more concise and general approach
using certain orthogonal polynomials is given in Section~\ref{sec-generalize}.

\section{Proof of the main result}
\label{sec-pf}

We wish to solve (\ref{system-fan}), whose coefficient matrix $M_G$ is close
to $M= nI + A'_3 = \theta_0 E_0 + \theta_1 E_1 + \theta_2 E_2$.
With $K$ denoting projection onto the kernel of $M$, we know that $M+\eta K$
is nonsingular for all  $\eta \neq 0$.  We begin by estimating its inverse
for a special choice of $\eta$.
\begin{lemma}
\label{inf-norm}
With $\eta^* = 2n$ and $A=M+\eta^* K$,
$$||A^{-1}||_\infty \le \frac{23}{9n} + O(n^{-2}).$$
\end{lemma}
\begin{proof}
Since the $E_i$ and $K$ are orthogonal idempotents for $\Omega(K_{n,n,n})$,
\begin{align}
\nonumber
A^{-1} &=  \theta_0^{-1} E_0 +  \theta_1^{-1} E_1 +  \theta_2^{-1}
E_2 + \eta^{-1} K\\
\label{inverse}
&= \eta^{-1} I + \sum_{j=0}^2 (\theta_j^{-1}-\eta^{-1}) E_j.
\end{align}
Substitute $\eta=\eta^*=2n$ and expressions for $\theta_j$ and $E_j$ from
Propositions~\ref{evals} and \ref{idempotents} into (\ref{inverse}).
Collect coefficients of the $A'_i$ to get
\begin{equation}
\label{leading-terms}
A^{-1} \approx  \frac{1}{n} A'_0 -\frac{1}{2n^2} A'_1 -\frac{4}{9n^3} A'_2
+  0 A'_3 -\frac{1}{18n^3} A'_4,
\end{equation}
where by `$\approx$' we mean that terms of lower degree in $n$ have been suppressed on each coefficient.
Apply the triangle inequality to (\ref{leading-terms}), making use of Proposition~\ref{degrees}, to get
\begin{align*}
||A^{-1} ||_\infty &\le \frac{1}{n} \nu_0
+\frac{1}{2n^2} \nu_1+ \frac{4}{9n^3}\nu_2 + 0 \nu_3 +\frac{1}{18n^3} \nu_4 + \text{lower terms}\\
&= \frac{1}{n} + \frac{2n}{2n^2} +
\frac{4n^2}{9n^3} + \frac{2n^2}{18n^3} + O(n^{-2}) =
\frac{23}{9n}+O(n^{-2}).\qedhere
\end{align*}
\end{proof}
By choice of $\eta^*$, we know that $A=M+ \eta^* K$ is real,
symmetric, and has all its eigenvalues at least $n$.  In particular, $A$ is
positive definite.
We next set up an application of Corollary~\ref{nonneg} to this $A$.

Suppose $G$ is a locally balanced 3-partite graph on $3n$ vertices with
$\delta(G) \ge 2(1-c)n$.  Let $A_G:=M_G+\eta^* K[G]$ as in Section~\ref{sec-sys} and define the perturbation
\begin{equation}
\label{perturb-defn}
\begin{tikzpicture}[scale=0.8,baseline=(current  bounding  box.center)]
\node at (-1,0) {$A+\delta A =~$};
\draw (0,-1.5)--(3.2,-1.5)--(3.2,1.5)--(0,1.5)--(0,-1.5);
\draw (0,-0.5)--(3.2,-0.5);
\draw (2.1,-0.5)--(2.1,1.5);
\node at (2.6,0.5) {$O$};
\node at (1,0.5) {$A_G$};
\node at (1.6,-1) {as in $A$};
\node at (3.4,0) {,};
\end{tikzpicture}
\end{equation}
where rows and columns are organized as edges of $G$ followed by edges of its 3-partite complement.  In particular, we take the `bottom' rows of $A+\delta A$ in this ordering to agree with those of $A$.  With this set-up, a solution of $(A+\delta A) \vx = \vj$ `restricts' to a solution of the smaller system $A_G \vx = \vj$.

The $(e,f)$-entry of $A[G]-A_G = M[G]-M_G$ records the number of triangles in $T(K_{n,n,n})$ which are missing in $T(G)$ and contain $e \cup f$.  Given any edge $e$ of $G$, at most $2cn$ edges of $K_{n,n,n}$ touching $e$ are missing in $G$.  Every triangle missing from $T(G)$ is counted in this way three times.  It follows that we have the bound $||\delta A||_\infty \le 6cn$.  So, already one has the estimate
$||A^{-1}\delta A||_\infty \le \frac{46c}{3}$ using submultiplicativity and Lemma~\ref{inf-norm}.  However, we can obtain a slightly better bound with some more work.

\begin{lemma}
\label{prodnorm}
With $\delta A$ and $c$ as defined above,
$||A^{-1} \delta A||_\infty \le \frac{40c}{3}+ O(n^{-1}).$
\end{lemma}

\begin{proof}
Begin by writing $\delta A = -(B_0 + B_3)$, where $B_0$ is a diagonal matrix
containing the main diagonal of $\delta A$.  Since entries of $\delta A$ arise from counting missing triangles in $G$, 
our matrices are integral and in fact satisfy the entrywise inequalities $O \le B_0 \le 2nc I$ and $O \le B_3 \le A'_3$.
Furthermore, a given edge of $G$ has two endpoints from which to join a third associate, 
and up to $2cn$ vertices in the other partite set define a missing triangle with it.  So $B_3$ has at most $4cn$ ones per row and column.

The key observation is that an entry of the product $A'_i B_3$ is simply
a partial count of the structure constants used for the product $A'_i A'_3$.  We
estimate the norm of the product supported on each $A'_h$ and identify some `cancellation'.  In particular, we focus on the term $h=3$.

Let $\circ$ denote entrywise product of matrices of the same shape.  Working from (\ref{leading-terms}), we compute
\begin{equation}
\label{ainv-e}
(A^{-1} \delta A) \circ A'_3 = \frac{1}{n} B_3 - \frac{1}{2n^2} (A_1' B_3) \circ A'_3 - \frac{1}{18n^3} (A'_4 B_3) \circ A'_3.
\end{equation}
Note that since structure constant $a_{23}^3$ vanishes, there is no contribution from the term $A'_2 B_3$.

We bound the rowsum of each $(A'_i B_3) \circ A'_3$ as follows.
Given an edge $e$ of $K_{n,n,n}$, we count the number of ordered pairs
$(f,g)$ of edges such that
\begin{itemize}
\item
$e$ and $f$ are $i$th associates
\item
$e$ and $g$ are 3rd associates; and
\item
$f$ and $g$ are 3rd associates defining a triangle in $K_{n,n,n}$ but not in $G$.
\end{itemize}
Consider $i=1$.  Given $e$, there are at most $2(n-1)$ choices for $f$ and, for each one, at most $2cn$ choices for $g$.  So $||A'_1 B_3 \circ A'_3||_\infty < 4cn^2$.  Now consider $i=4$.  Given $e$, there are at most $2(n-1)$ choices for one vertex of $f$ and subsequently at most $cn$ choices for the second vertex (which also uniquely determines $g$).  Thus $||A'_4 B_3 \circ A'_3||_\infty < 2cn^2$.  Considering again (\ref{ainv-e}) and noting the opposite signs of terms, we have
\begin{equation}
\label{savings}
||(A^{-1} \delta A) \circ A'_3||_\infty \le \max\left\{ \tfrac{1}{n} 4cn, \tfrac{1}{2n^2} 4cn^2+ \tfrac{1}{18n^3} 2cn^2 \right\} \le 4c.
\end{equation}
(This is a savings from $6c+O(n^{-1})$ that would arise from the triangle inequality.)  We did not identify any opposite signs in the expansion of remaining terms.  So there is no loss in estimating the remaining eight terms of $A^{-1}\delta A$ using the triangle inequality and submultiplicativity; this leads to
\begin{equation*}
||A^{-1} \delta A ||_\infty \le \sum_{j=0}^4 ||(A^{-1} \delta A) \circ A'_j||_\infty =
2c + 2c + \frac{10c}{9} + 4c + \frac{38c}{9} + O(n^{-1}) = \frac{40c}{3}+ O(n^{-1}).
\qedhere
\end{equation*}
\end{proof}

Let $c < 3/80$ and let $n$ be large.  Invoke Corollary~\ref{nonneg} with $A$ and $\delta A$ as
described.  The conclusion is that $A+\delta A$ (and hence $A_G$, by construction) is invertible.  Moreover, the solution vector $A_G^{-1} \vj$ is entrywise nonnegative.  By the set-up in Section~\ref{sec-sys}, this vector defines the weights of a fractional fan-decomposition of $G$.  

Taking $c$ approaching $3/80$, we have the decomposition degree threshold $\tau_F(3) \le 77/80 = 0.9625$.  

Now, a small extra improvement is possible following a method of Garaschuk \cite[Chapter 4]{Kseniya}.
First, we inspect the proof of Lemma~\ref{prodnorm} and note that the sum of positive entries in any row of $A^{-1} \delta A$ is at most $28c/3+ O(n^{-1})$.  (Removing the  $A'_0 B_0$ term saves $2c$ and the negative part of (\ref{savings}) saves another $2c+O(n^{-1})$.)  Using an extra term of the series expansion for Proposition~\ref{cond}, we have the entrywise inequality
\begin{equation*}
\vx+\delta \vx = (I+A^{-1} \delta A)^{-1}A^{-1} \vj  \ge \frac{1}{3n} \left( \vj-(A^{-1} \delta A) \vj   - \sum_{i=2}^\infty ||A^{-1} \delta A||_\infty  \vj \right).
\end{equation*}
That is, the solution vector is nonnegative for large $n$ provided 
$$\frac{28c}{3} + \frac{(40c/3)^2}{1-40c/3} < 1,$$
or $c < (\sqrt{409}-17)/80 \lessapprox 0.04$.  So in fact $\tau_F(3) < 0.96$.

\section{Larger cliques and hypergraphs}
\label{sec-generalize}

In this section, we sketch how our method extends to larger cliques and
hypergraphs in the multipartite setting.  Specifically, let $G$ be a
$k$-partite $t$-uniform hypergraph with $n$ vertices in each partite set.
To be clear, edges consist of at most one vertex in each partite set.
Further, suppose $G$ is \emph{locally balanced} in the following sense: any
$t-1$ vertices in distinct partite sets are together in an edge with equally
many vertices in each of the other partite sets.  Finally, we assume these
neighborhoods are close to full: $\delta_{t-1}(G) \ge (1-c) (k-t+1)n$.  We
investigate thresholds on $c$ sufficient for the fractional
$K_k^t$-decomposition of such hypergraphs $G$.

The question for exact decompositions is challenging even for $c=0$.  Let
$K[t,k,n]$ denote the complete balanced $k$-partite $t$-graph on $kn$
vertices.  A $K_k^t$-decomposition of $K[t,k,n]$ is equivalent to an
orthogonal array $OA[t,k,n]$, also known as a `transversal design'.

Before continuing, we offer some clarifying remarks on notation.  In
Section~\ref{sec-intro} and in references \cite{BKLO,BKLOT}, the parameter
$r$ is used for clique size.  Moreover, in \cite{BKLMO}, $k$ is used for
hypergraph rank.  Note the different notation here, which we chose for consistency
with the underlynig coding theory and design theory.  Next, in
Sections~\ref{sec-assoc} and \ref{sec-pf} we primarily used the the second
level of the Hamming lattice.  Here, the treatment is more general and we
express everything in terms of the top level; that is, we work exclusively
in the Hamming scheme $H(k,n)$.

\subsection{Spectral computations in $H(k,n)$}

Let $A_0,A_1,\dots,A_k$  and $E_0,E_1,\dots,E_k$ be the adjacency matrices
and orthogonal idempotents of $H(k,n)$.  They are related via
\begin{equation}
\label{hamming-eigs}
A_i = \sum_{j=0}^k \kappa_i(j) E_j~~\text{and}~~E_j = \frac{1}{n^k}
\sum_{i=0}^k \kappa_j(i) A_i,
\end{equation}
where $$\kappa_i(x) = \sum_{l} (-1)^l (n-1)^{i-l} \binom{k-x}{i-l}
\binom{x}{l}$$ is the \emph{Krawtchouk} polynomial of degree $i$.  See
\cite[Chapter 30]{wvl2}, for instance.

Let $W$ denote the inclusion matrix of the $t$th level of $H(k,n)$ versus
the top level.  Then, as before, $M=WW^\top$ stores in its $(e,f)$-entry the
number of $k$-cliques containing both $e$ and $f$ in $K[t,k,n]$.

\begin{prop}
\label{hyp-evals}
The nonzero eigenvalues of $M$ are $\theta_j =\binom{k-j}{k-t} n^{k-t}$,
with multiplicity $\binom{k}{j} (n-1)^j$,
$j=0,1,\dots,t$.
\end{prop}

\begin{proof}
Instead of $M=WW^\top$, it suffices to compute the nonzero eigenvalues of
\begin{equation*}
W^\top W = \sum_{i=0}^k \binom{k-i}{t} A_i = \sum_{j=0}^k E_j
\sum_{i=0}^{k-t} \binom{k-i}{t} \kappa_i(j).
\end{equation*}
Recall that the $E_j$ are orthogonal idempotents.  It follows that
eigenvalues are given by the inner sum, call it $\theta_j$, with
corresponding multiplicities
$\text{rank}(E_j) = \binom{k}{j}(n-1)^j$.
It remains to simplify $\theta_j$.

The value $\kappa_i(j)$ is the coefficient of $X^{i}$ in
$(1+(n-1)X)^{k-i}(1-X)^{j}$.
So, our sum equals the coefficient of $X^{k-t}$ in
\begin{equation}
\label{gen-function}
(1+(n-1)X)^{k-j}(1-X)^j \sum \binom{t+i}{t}
X^i=(1+(n-1)X)^{k-j}(1-X)^{j-t-1}.
\end{equation}
For $j>t$, (\ref{gen-function}) is a polynomial of degree $k-t-1$, and so the
coefficient of $X^{k-t}$ vanishes.  For $0 \le j \le t$, we compute
\begin{align*}
\theta_j &= \sum_{l=0}^{k-t} (-1)^{l} (n-1)^{k-t-l} \binom{k-j}{k-t-l}
\binom{j-t-1}{l}  \\
&= \sum_{l=0}^{k-t} (n-1)^{k-t-l} \binom{k-j}{t-j+l} \binom{t-j+l}{l} \\
&= \binom{k-j}{k-t} \sum_{l=0}^{k-t} (n-1)^{l} \binom{k-t}{l} =
\binom{k-j}{k-t} n^{k-t}. \qedhere
\end{align*}
\end{proof}

Next, consider $E_j':=\theta_j^{-1} WE_jW^\top$.  Since $ME_j' =
\theta_j^{-1} W(W^\top W E_j) W^\top =   WE_j W^\top = \theta_j E_j'$ and
$$(E_j')^2 = \theta_j^{-2} WE_j W^\top W E_j W^\top = \theta_j^{-1} WE_j^2
W^\top = E_j',$$
it follows that $E_j'$ is projection onto the eigenspace of $M$
corresponding to eigenvalue $\theta_j$.

We now compute, using  (\ref{hamming-eigs}),
\begin{align}
\nonumber
(M+\eta K)^{-1} &=\sum_{j=0}^t \frac{1}{\theta_j} E_j' + \frac{1}{\eta}
\left( I-\sum_{j=0}^t E_j' \right) \\
\nonumber
 &= \frac{1}{\eta} I + \sum_{j=0}^t \frac{1}{\theta_j} \left(
\frac{1}{\theta_j} - \frac{1}{\eta} \right) WE_jW^\top\\
\label{hyp-inverse}
&= \frac{1}{\eta} I + \frac{1}{n^k} \sum_{i=0}^k \sum_{j=0}^t
\frac{1}{\theta_j} \left( \frac{1}{\theta_j} - \frac{1}{\eta} \right)
\kappa_j(i) WA_iW^\top.
\end{align}

To go further, one must study the matrices $WA_i W^\top$, $i=0,1,\dots,k$.
It is easy to see that, for edges $e$ and $f$,
the $(e,f)$-entry of $WA_i W^\top$ equals the number of ordered pairs $(a,b)
\in [n]^k \times [n]^k$ such that $a$ extends $e$, $b$ extends $f$, and
where $a$ and $b$ are at Hamming distance $i$ in $H(k,n)$.

\begin{prop}
\label{crs-waw}
The constant row sum of $WA_iW^\top$ is
$\begin{displaystyle}\binom{k}{t} \binom{k}{i} n^{k-t} (n-1)^i
\end{displaystyle}.$
\end{prop}
\begin{proof}
There are $n^{k-t}$ extensions of $e$ to a $k$-tuple $a$.  For each such
extension, there are $\binom{k}{i} (n-1)^{i}$ choices for a tuple $b$ at
Hamming distance $i$.  Finally, there are $\binom{k}{t}$ restrictions of $b$
to an edge $f$.
\end{proof}

\subsection{Estimates for $t=2$}

Here, we consider the graph case with general clique size and sketch a norm
bound.  We begin with a more detailed version of the counting argument in
Proposition~\ref{crs-waw}.  The second level of $\mathcal{H}_{k,n}$ has six
relations (and corresponding adjacency matrices): identical edges
$(A'_0=I)$,  adjacent unequal edges in the same pair of partite classes
$(A'_1)$, disjoint edges in the same pair of partite classes $(A'_2)$,
adjacent edges touching exactly three partite classes $(A'_3)$, disjoint
edges touching exactly three partite classes $(A'_4)$, and disjoint edges
touching four partite classes $(A'_5)$.  Refer to Figure~\ref{rel-labels}.

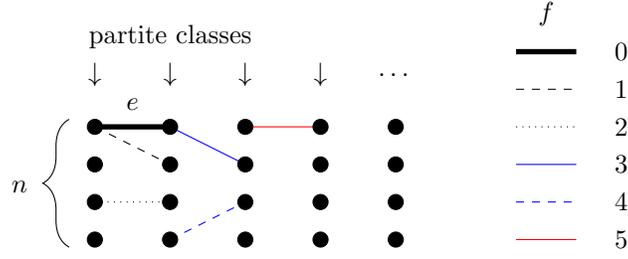
\begin{figure}[htpb]
\begin{tikzpicture}
 \draw[line width=2pt] (0,1.5)--(1,1.5);
\draw[dashed] (0,1.5)--(1,1);
\draw[dotted] (0,.5)--(1,.5);
\draw[blue] (1,1.5)--(2,1);
\draw[dashed, blue] (1,0)--(2,.5);
\draw[red] (2,1.5)--(3,1.5);

\node at (1,2.7) {partite classes};
\node at (0,2.2) {$\downarrow$};
\node at (1,2.2) {$\downarrow$};
\node at (2,2.2) {$\downarrow$};
\node at (3,2.2) {$\downarrow$};
\node at (4,2.2) {$\dots$};

\draw [decorate,decoration={brace,amplitude=10pt},xshift=-4pt,yshift=0pt]
(-.2,-.1) -- (-.2,1.6);

\node at (-1,.7) {$n$};
\node at (0.5,1.8) {$e$};

\filldraw (0,0) circle [radius=.1];
\filldraw (0,0.5) circle [radius=.1];
\filldraw (0,1) circle [radius=.1];
\filldraw (0,1.5) circle [radius=.1];

\filldraw (1,0) circle [radius=.1];
\filldraw (1,0.5) circle [radius=.1];
\filldraw (1,1) circle [radius=.1];
\filldraw (1,1.5) circle [radius=.1];

\filldraw (2,0) circle [radius=.1];
\filldraw (2,0.5) circle [radius=.1];
\filldraw (2,1) circle [radius=.1];
\filldraw (2,1.5) circle [radius=.1];

\filldraw (3,0) circle [radius=.1];
\filldraw (3,0.5) circle [radius=.1];
\filldraw (3,1) circle [radius=.1];
\filldraw (3,1.5) circle [radius=.1];

\filldraw (4,0) circle [radius=.1];
\filldraw (4,0.5) circle [radius=.1];
\filldraw (4,1) circle [radius=.1];
\filldraw (4,1.5) circle [radius=.1];

\node at (6,3) {$f$};

\draw[line width=2pt] (5.6,2.5)--(6.4,2.5);
\draw[dashed] (5.6,2)--(6.4,2);
\draw[dotted] (5.6,1.5)--(6.4,1.5);
\draw[blue] (5.6,1)--(6.4,1);
\draw[dashed, blue] (5.6,0.5)--(6.4,0.5);
\draw[red]  (5.6,0)--(6.4,0);

\node at (7,2.5) {$0$};
\node at (7,2) {$1$};
\node at (7,1.5) {$2$};
\node at (7,1) {$3$};
\node at (7,.5) {$4$};
\node at (7,0) {$5$};
\end{tikzpicture}

\caption{Relation labels for edges in the second level of
$\mathcal{H}_{k,n}$}
\label{rel-labels}
\end{figure}

In what follows, define
$$F_s(h,i) := n^h (n-1)^i \sum_{l=0}^{2s} \binom{h}{i-l} \binom{2s}{l} =
\binom{h+2s}{i} n^{h+i} + O(n^{h+i-1}).$$

\begin{prop}
\begin{align*}
WA_{i}W^\top &= F_0(k-2,i) A'_{0} + F_0(k-2,i-1) A'_{1} + F_0(k-2,i-2)A'_{2}
\\
& + F_1(k-3,i) A'_3 + F_1(k-3,i-1) A'_4 + F_2(k-4,i) A'_5.
\end{align*}
\end{prop}

\begin{proof}
Consider the $(e,f)$-entry of $WA_iW^\top$, where $e$ and $f$ are $j$th
associates, $j=0,1,2$.  In this case, $e \cup f$ touches only two partite
classes, so there are $n^{k-2}$ choices for an extension $a$ of $e$ to a
$k$-tuple.  Next, choose which $i-j$ of the $k-2$ newly added vertices to
change in an extension $b$ of $f$ to a $k$-tuple at Hamming distance $i$ from $a$.
Finally, choose any of the other $n-1$ vertices in each corresponding
partite set.  The total count is $n^{k-2} (n-1)^{i-j} \binom{k-2}{i-j} =
F_0(k-2,i-j)$.

Now consider the case in which $e$ and $f$ are $3$rd associates, say
$e=\{u,v\}$ and $f=\{v,w\}$, where $u,v,w$ are in distinct partite classes.
The counting is similar as above, but divides into cases according to which of
$u,v,w$ are common to both extensions $a$ of $e$ and $b$ of $f$.  If $a$ and
$b$ agree on all three points, there is a choice of $i$ other partite sets for disagreements, leading to a count of 
$n^{k-3} (n-1)^i \binom{k-3}{i}$.  If they agree on $v$ and exactly one of $u,w$, there are only $i-1$ other disagreements 
and the binomial coefficient changes accordingly.  If they agree only on $v$, there are $i-2$ other disagreements.  The total count is $F_1(k-3,i)$, and a
very similar argument obtains $F_1(k-3,i-1)$ for $4$th associates.

Finally, suppose $e$ and $f$ are $5$th associates, meaning $e \cup f$
touches four distinct partite classes.  If $a$ and $b$ disagree on exactly
$l$ of these four partite classes, there are $n^{k-4} (n-1)^i
\binom{k-4}{i-l}$.  Summing over the possible cases for disagreements, we
obtain $F_2(k-4,i)$.
\end{proof}

Now, we work from (\ref{hyp-inverse}) and begin by analyzing degrees of the
polynomial terms in $n$.  The dominant terms occur for $i+j \ge k$, of which
there are only six pairs $(i,j)$.  Moreover, similar to
Section~\ref{sec-pf}, we put $A= M+ \eta^* K$, where $\eta^* = \theta_1$.  This causes
the terms for $j=1$ to vanish, leaving only four remaining terms in the cases
$$(i,j) \in \{(k,0),(k,2),(k-1,2),(k-2,2)\}.$$
Substitute $\kappa_0(k) = 1$, $\kappa_2(k) = \binom{k}{2}$, $\kappa_2(k-1) =
(1-k)n+O(1)$,  $\kappa_2(k-2) = n^2 + O(n)$, $\theta_0=\binom{k}{2}
n^{k-2}$, $\theta_1 = (k-1)n^{k-2}$, $\theta_2=n^{k-2}$, and collect
dominant terms of the coefficients of the $A'_h$.
After some calculations and the triangle inequality, we have
\begin{align*}
||A^{-1}||_\infty
& \lessapprox n^{2-k} \nu_0 + \tfrac{k-2}{k-1} n^{1-k}\nu_1  +
\left(\tfrac{k-2}{k-1} -d_k \right) n^{-k}\nu_2 + 0\nu_3 + d_k n^{-k} \nu_4
+ d_k n^{-k}\nu_5
\end{align*}
plus terms of lower order,
where $d_k := \frac{k-2}{2\binom{k}{2}^2}$, and $\nu_i$ is the row sum of $A'_i$.
Finally, substitute $\nu_0 = 1$, $\nu_1 =2(n-1)$, $\nu_2=(n-1)^2$,
$\nu_3=2(k-2)n$, $\nu_4=2(k-2)n(n-1)$, and $\nu_5=\binom{k-2}{2}n^2$ to
obtain
\begin{equation}
\label{norm-bound-cliques}
||A^{-1}||_\infty  \le \left(4-\frac{k^3+k-4}{2 \binom{k}{2}^2}
\right) n^{2-k} + O(n^{1-k}).
\end{equation}
In the case $k=3$, note that the formula for $\nu_5$ vanishes.  So the same
formula recovers the leading coefficient $23/9$ from Section~\ref{sec-pf}.
Leading coefficients for more small values of $k$ are given in
Table~\ref{coeffs-smallk}.
\begin{table}[hpbt]
$$\begin{array}{r|cccc}
k & 3 & 4 & 5 & 6\\
\hline \\[-8pt]
\text{leading coeff} & \dfrac{23}{9} & \dfrac{28}{9} & \dfrac{337}{100} &
\dfrac{791}{225}
\end{array}$$
\medskip
\caption{Leading coefficients of $||A^{-1}||_\infty$ for small $k$}
\label{coeffs-smallk}
\end{table}

As in the case of triangles, the matrix $A_G$ we wish to consider is close
to the restriction $A[G]$ of $A$. Assume $\delta(G) \ge (1-c)(k-1)n$ and set up the perturbation $A+\delta A$ as in (\ref{perturb-defn}).  By counting missing cliques as in Section~\ref{sec-pf}, 
\begin{equation}
\label{hyper-perturbnorm}
||\delta A||_\infty \le ||M[G]-M_G||_\infty < c \binom{k}{2}^2 n^{k-2} +  O(n^{k-3}).
\end{equation}
After (the submultiplicativity variant of) Corollary~\ref{nonneg}, we obtain a fractional fan decomposition threshold $\tau_F(k)
\lessapprox 1-1/2k^4$.  For large $k$, a better bound has been
obtained by Montgomery in \cite{Montgomery}.  However, our method leads to
reasonably good thresholds for small $k$.  In the case $k=4$, for instance, every edge
belongs to at most $4cn^2$ missing cliques in $G$.  This leads to an error norm of $24cn^2$.  Together with
the entry from Table~\ref{coeffs-smallk}, we get $\tau_F(4) \le 1-1/(2 \cdot 24 \cdot 28/9) = 445/448$.
When used in conjunction with Theorem~\ref{integralize}, this gives a
result on completion of partially filled orthogonal latin squares.  We omit the extra small improvements that were obtained for triangles in Section~\ref{sec-pf}.

\subsection{Estimates for general $t$}
\label{sec-hyper}

In the case $t=2$, we estimated the norm of $(M+\eta K)^{-1}$ by expanding
each $WA_iW^\top$ in the second level of $\mathcal{H}_{k,n}$.  Such an
expansion becomes more involved for $t>2$.  However, it is possible to get a
crude bound working from (\ref{hyp-inverse}) using only Proposition~\ref{crs-waw} and the triangle inequality.  We have
\begin{equation}
\label{hyper-opteta}
||(M+\eta K)^{-1}||_\infty \le \frac{1}{\eta} + \frac{1}{n^t} \binom{k}{t}
\sum_{i=0}^k \binom{k}{i} (n-1)^i \left| \sum_{j=0}^t \frac{1}{\theta_j}
\left( \frac{1}{\theta_j} - \frac{1}{\eta} \right)  \kappa_j(i) \right|.
\end{equation}
Optimizing the inner sum over $\eta$ is tricky, but it simplifies for large
$\eta$ as
\begin{align*}
\lim_{\eta \rightarrow \infty} ||(M+\eta K)^{-1}||_\infty &\le \frac{1}{n^t}
\binom{k}{t} \sum_{i=0}^k \binom{k}{i} (n-1)^i \left| \sum_{j=0}^t
\frac{\kappa_j(i)}{\theta_j^2} \right|\\
&\le \frac{1}{n^t} \binom{k}{t} \sum_{j=0}^t \sum_{i=0}^k \binom{k}{i}
(n-1)^i \theta_j^{-2} |\kappa_j(i)|.
\end{align*}
Observe that $\kappa_j(i)$ is a polynomial of degree $\min\{j,k-i\}$ in
$n$.  It follows that the dominant terms on the right occur when $i=k-j$.
Substituting $\kappa_j(k-j)=n^j+O(n^{j-1})$ and for  $\theta_j$ using
Proposition~\ref{hyp-evals}, 
\begin{align}
\nonumber
\lim_{\eta \rightarrow \infty} ||(M+\eta K)^{-1}||_\infty
&\le \frac{1}{n^t} \binom{k}{t} \sum_{j=0}^t \frac{\binom{k}{k-j}
(n-1)^{k-j} [n^j +O(n^{j-1})]}{\binom{k-j}{k-t}^2 n^{2k-2t}} \\
\nonumber
&\le n^{t-k} \binom{k}{t} \sum_{j=0}^t \binom{k}{j} \binom{k-j}{k-t}^{-2}  +
O(n^{t-k-1})\\
\label{drop-term}
&\le    2^t \binom{k}{t}^2 n^{t-k} +  O(n^{t-k-1}).
\end{align}
Note that (\ref{drop-term}) equals the line above when the exponent in the
sum is changed from $-2$ to $1$ (after some binomial identities are
applied).  For given specific $k$ and $t$, it is not difficult to compute a better
constant.  In any case, there exists $C(t)$ so that, for some $\eta^*$,
\begin{equation}
\label{hyper-infnorm}
||(M+\eta^* K)^{-1}||_\infty < C(t) \binom{k}{t}^2 \frac{1}{n^{k-t}} +
O\left(\tfrac{1}{n^{k-t+1}}\right).
\end{equation}
Under the assumption $\delta_{t-1}(G) \ge (1-c)(k-t+1)n$, it is not
difficult to imitate \cite[Proposition 3.3]{Dukes-frac} and obtain
\begin{equation}
\label{hyper-perturbnorm}
||M[G]-M_G||_\infty < c \binom{k}{t}^2 n^{k-t} +  O(n^{k-t+1}).
\end{equation}
From (\ref{hyper-infnorm}), (\ref{hyper-perturbnorm}) and Corollary~\ref{nonneg}, one obtains a
threshold on the allowed missing degree proportion $c$ which is of the order
$k^{-4t}$.  In many cases, it may be possible to do better, especially if
a bound before (\ref{drop-term}) is computed.  Even
still, this threshold is likely to be significantly improved through other
methods.  For this reason, we omit a detailed
treatment in the general setting.  Besides, the stakes are lower since there is
presently no analog of Theorem~\ref{integralize} for hypergraphs.

\section{Discussion}
\label{sec-discuss} 

In spite of the interesting algebra connected with our matrix for $K_{n,n,n}$, the approximation via linear perturbation 
probably incurs considerable loss.  On the other hand, it is noteworthy that this method still delivers a reasonable threshold guaranteeing a decomposition.  And, in practice, simply solving our linear system (\ref{system-fan}) stands a good chance at giving a fractional decomposition, even if the guarantee is not met.

By contrast, the methods in \cite{BKLMO,Dross,Montgomery} use
local adjustments to an initial constant weighting of cliques.  For example,
the paper \cite{BKLMO} of Barber, K\"uhn, Lo, Montgomery and Osthus, which studies fractional decomposition of 
dense graphs and hypergraphs, uses the fact that an edge $e$ can be expressed as a linear combination of the $r$-cliques inside of an $(r+2)$-clique containing $e$.  Averaging over many such $(r+2)$-cliques, the authors obtain an `edge
gadget' which adjusts the weight of $e$ via a minor change to the weighting
of cliques.  In the $r$-partite setting, it is not immediately clear how to construct gadgets.  However,
Montgomery overcomes this challenge in the recent paper \cite{Montgomery}.  There, 
$\tau_F(r) \le 1-10^{-6}r^{-3}$ is obtained for the general $r$-partite setting.  
This improves on our exponent by one, but the small constant illustrates the extra difficulty with local adjustments in this setting.

One interesting feature common to most work on fractional decomposition is that results are stated in terms of minimum (vertex or co-) degree.  However, perhaps a more natural hypothesis is the (slightly weaker) condition that every edge belong to many cliques.  In general, it would be of interest to explore decompositions under different hypotheses.

A generalization we have not considered is the (fractional)
$K_s$-decomposition of complete $r$-partite graphs for $s \le r$.  Of
course, given a $K_r$-decomposition, it is possible to replace each
$r$-clique with a scaled average of $s$-cliques, but then the minimum degree
threshold will depend on $r$ rather than $s$.  When $r \ge s+2$, Montgomery
has observed that the gadget technique of \cite{BKLMO} can produce
reasonable thresholds for this problem that depend on $s$.

For hypergraphs in the partite setting, the problem is still in early stages.  
Our outline in Section~\ref{sec-hyper} offers a starting point for this problem.

In the same way that triangle decomposition of $3$-partite graphs models
completion of partial latin squares, the decomposition problem for general
clique size leads to mutually orthogonal latin squares (MOLS).  See \cite[\S
3.2]{BKLOT} for more discussion on this.  As one very special case of this
problem, the main result of \cite{Dukes-VB} shows that there exists $r$ MOLS
of order $n$ missing a `hole' of order $m$ for all large $n$ and $m$ satisfying $n
\ge 8(r+1)^2 m$.  This corresponds with missing degree proportion of order
$r^{-2}$.  An improvement to order $r^{-1}$ in this special case would be of
some interest for design theorists.

We close with some remarks on convex-geometric barriers for our problem.
A locally balanced 3-partite graph on $3n$ vertices admits a fractional triangle decomposition if and only if it belongs to the cone of weighted graphs generated by triangles in $K_{n,n,n}$.  The facet structure of this cone (its description by inequalities) is perhaps of some interest in its own right.  For instance, in the case $n=2$, a weighted graph belongs to the cone only if twice the sum of edge weights on two disjoint triangles exceeds the sum of edge weights crossing between them.  This inequality defines one of 16 distinct facets of the cone for $n=2$.  We have computed $207$ distinct facets for $n=3$, arising from four isomorphism classes, and $113740$ distinct facets for $n=4$, falling into $15$ isomorphism classes.  More precisely, these classes are orbits under the action of $\text{Aut}(K_{n,n,n}) = \mathcal{S}_n \wr \mathcal{S}_3$.  It is clear from these experiments that the cone is very complex; however, even a partial description may lead to a better understanding of geometric barriers for the fractional decomposition problem.

\section*{Acknowledgments}

We are grateful to the authors of \cite{BKLOT} and \cite{Montgomery} for providing helpful remarks and updates on their recent
work.

\end{document}